\documentclass{article}
\usepackage{amsfonts}
\usepackage{amssymb}
\usepackage[mathcal]{eucal}
\usepackage{mathrsfs}
\usepackage{latexsym}
\usepackage{amsthm}
\usepackage{amsmath,amscd}
\usepackage{enumerate}
\usepackage{hyperref}
\long\def\symbolfootnote[#1]#2{\begingroup\def\thefootnote{\fnsymbol{footnote}}\footnote[#1]{#2}\endgroup}

\newcommand{\Z}{\mathbb{Z}}
\renewcommand{\P}{\mathscr{P}}

\newcommand{\Pri}{\ensuremath{\mathbb{P}}}
\theoremstyle{plain}
\newtheorem{theorem}{Theorem}[section]

\newtheorem{proposition}[theorem]{Proposition}
\newtheorem{corollary}[theorem]{Corollary}

\theoremstyle{definition}
\newtheorem{definition}[theorem]{Definition}

\newcommand{\sub}{\subseteq}
\newcommand{\sm}{\smallsetminus}

\renewcommand{\phi}{\varphi}

\newcommand{\ba}{boolean algebra}

\newcommand{\pa}[1]{\left(#1\right)}

\newcommand{\s}[1]{\left\{#1\right\}}

\newcommand{\defeq}{\ensuremath{\stackrel{\text{\tiny def}}{=}}}
\renewcommand{\epsilon}{\varepsilon}

\newcommand{\PoF}{\ensuremath{\P\left(\omega\right)/{\mathrm {fin}}}}

\begin{document}
\title{A note on maximality of ideal-independent sets}
\author{Corey~T. Bruns, University of Wisconsin-Whitewater, brunsc@uww.edu}


\maketitle
\section{Introduction and Definitions}
We give a solution to a problem originally posed in a draft of J.D. Monk \cite{MR2387943}.  This result is now listed as further fact 4 after theorem 2.16.
In doing so, we will consider some subsets of boolean algebras.  

We will follow S. Koppelberg \cite{Handbook} for notation.  In particular $+$, $\cdot$, and $-$ will used as the boolean operators, and $0$ and $1$ as the least and greatest element.  By extension, the least upper bound of a set $M$ will be denoted $\sum M$.

\begin{definition}A subset $X$ of a boolean algebra is ideal-independent if $0,1\notin X$ and $\forall x\in X,
x\not\in\left<X\sm\s x\right>^{id}$; equivalently, for distinct $x,x_1,\ldots,x_n\in X$, $x\not\leq x_1+x_2+\ldots +x_n$. 
\end{definition}

By Zorn's lemma, there are maximal ideal-independent sets.

This is used in one of several equivalent definitions in Monk \cite{CINV2} of the spread $s$ of a boolean algebra $A$:  
$$s\pa A=\sup\s{\left|X\right| : X\mbox{ is ideal-independent.}}$$

Related to this cardinal function is the minimaximal version: $$s_{mm}\pa A \defeq\min\s{\left|X\right|: X\mbox{ is infinite and maximal ideal-independent}}.$$
Among the results of Monk \cite{MR2387943} is that it is consistent with ZFC that $${\aleph_0<s_{mm}\pa\PoF<\beth_1}.$$  We will further show how $s_{mm}$ compares to the pseudo-intersection number $\mathfrak p$.

\section{Maximality}

Our main result is a necessary condition for maximality of an ideal-inde\-pendent set. 

\begin{theorem}\label{thm_idindsum1}Let $A$ be a \ba{}. If $X\sub A$ is maximal for
ideal-independence, then $\sum X=1$.
\end{theorem}
\begin{proof}
Let $X\sub A$ be ideal-independent and let $b\in A\sm\s 1$ be such that
$\forall x\in X x\leq b$.  We claim that $X'\defeq X\cup\s{-b}$ is ideal-independent
(thus proving the theorem).

We need to show that for distinct $x,x_1,x_2,\ldots,x_n\in X'$, $x\not\leq
x_1+x_2+\ldots+x_n$.  There are three cases to consider.

\begin{enumerate}
\item $-b\not\in\s{x,x_1,x_2,\ldots,x_n}$.

Then $x\not\leq x_1+x_2+\ldots+x_n$ by the ideal-independence of $X$.

\item $-b=x$.

Assume otherwise, that is, $-b\leq x_1+x_2+\ldots+x_n$.
Since $x_1+x_2+\ldots+x_n\leq b$, we have $-b\leq b$, thus $b=1$,
contradiction.

\item $-b=x_1$.

Then as $x\in X$, $x\leq b$, that is, $x\cdot -b = 0$.
By the ideal-independence of $X$, $x\not\leq x_2 + x_3 + \ldots + x_n$, so that
$y\defeq x\cdot\pa{x_2 + x_3 + \ldots + x_n}<x$.

Then $${x\cdot\pa{-b + x_2 + x_3 + \ldots + x_n}=x\cdot -b + x\cdot \pa{x_2 + x_3
+ \ldots + x_n}=0+y<x,}$$ thus $x\not\leq -b + x_2 + x_3 + \ldots + x_n$.
\end{enumerate}
\end{proof}


It is worth noting that a converse of Theorem \ref{thm_idindsum1} does not hold for infinite sets.
That is, there is an ideal-independent set $X$ with $\sum X=1$ that is not maximal.

\begin{proposition}Let $\Pri=\s{p\in \omega:p \mbox{ is prime}}$.  In the boolean algebra \\${\P(\omega\sm\s{0,1})}$ where $+$ is union, $\cdot$ is intersection, and $-x$ is $\pa{\omega\sm\s{0,1}}\sm x$, set $X=\s{ p\Z^+ : p\in\Pri}$ (where $n\Z^+$ is the set of all nonzero multiples of n).
Then all of the following are true:
\begin{enumerate}

\item $\sum X=1$ 

\item $X$ is an ideal-independent set.

\item $X\cup \s{\Pri}$ is also ideal-independent.
\end{enumerate}
\end{proposition}
\begin{proof}
\begin{enumerate}

\item $\bigcup X=\omega\sm\s{0,1}$ as every integer other than $0$ and $1$ is a nonzero multiple of a prime, so $\sum X=1$.

\item If $p$ is prime, then $p$ is not the multiple of any other prime, so $p\Z^+\not\leq  p_1\Z^+ + p_2\Z^+ +\ldots + p_n\Z^+$ if $p,p_1,\ldots,p_n$ are all distinct primes.

Thus $X$ is ideal-independent.

\item We have two parts to show that $X\cup \s{\Pri}$ is ideal-independent; Let $Y$ be a finite subset of $X$ and $q,r$ distinct primes such that $q\Z^+\notin Y$ and $r\Z^+\notin Y$.  We must show that $\Pri\not\leq \sum Y$ and that $q\Z^+\not\leq \Pri+\sum Y$.

$q\in\Pri$ and $q\notin \sum Y$, so the first is true.

$rq\in q\Z^+$, but $rq\notin \Pri$ as it is composite and $rq\notin\sum Y$ since neither $r\Z^+\in Y$ nor $q\Z^+\in Y$.

Thus $X\cup\s\Pri$ is ideal-independent and so $X$ is not maximal. 

\end{enumerate}
\end{proof}

\begin{corollary}
For $A$ infinite, $\mathfrak p\pa A \leq s_{mm}\pa A$; in particular\\ $\mathfrak p=\mathfrak p\pa{\PoF}\leq s_{mm}\pa{\PoF}$.
\end{corollary}

We recall from Monk \cite{CINV2} that the pseudo-intersection number $\mathfrak p$ is defined as:
$$\mathfrak p \pa A \defeq \min\s{\left|Y\right|:\sum Y = 1 \;\mbox{and}\; \sum Y'\neq 1\;
\mbox{for every finite}\; Y'\sub Y}.$$

\begin{proof}
Let $X$ be maximal ideal-independent (and thus infinite).  By Theorem \ref{thm_idindsum1}, $\sum
X=1$. 

If $X'\sub X$ is finite, $\sum X'\neq 1$, otherwise, take $x\in
X\sm X'$, then $x\leq 1 = \sum X'$, contradicting the ideal-independence of
$X$.  Thus $$\s{\left|X\right|: X\mbox{ is
maximal ideal-independent}}\sub$$ $$\s{\left|Y\right|:\sum Y = 1 \;\mbox{and}\; \sum Y'\neq 1\;
\mbox{for every finite}\; Y'\sub Y}$$ and so $\mathfrak p\pa A \leq s_{mm}\pa
A.$
\end{proof}

Now, $\mathfrak p$ is one of the smaller continuum cardinals, so that $s_{mm}$ is larger is not particularly surprising. However, the method of proof does turn out to be useful with other similar properties. Other set properties involving sums of finite subsets, such as independence and n-independence, are amenable to these methods. In \cite{DisCh12}, we introduce some variations on $\mathfrak i$; their placement in order of continuum cardinals can be similarly determined.

\bibliographystyle{spmpsci}
\bibliography{biblio}
\end{document}